\documentclass[final
]{elsarticle
}

\RequirePackage[abspath,realmainfile]{currfile}

\usepackage{amsmath,amssymb,amsthm,amsfonts,amsbsy}
\usepackage{graphicx}
 
\usepackage{enumerate}
\usepackage
{showkeys}

\usepackage{datetime} 
\usepackage{xcolor}

\usepackage{mathrsfs}

\usepackage{hyperref}

\theoremstyle{plain} 
\newtheorem{theorem}{Theorem}
\newtheorem{corollary}
{Corollary}

\newtheorem*{theoremB*}{Theorem B}
\newtheorem*{theorem1.1ep*}{Theorem 1.$\boldsymbol{\vp}$}
\newtheorem*{theorem1.2ep*}{Theorem 2.$\boldsymbol{\vp}$}
\newtheorem{lemma}
{Lemma}

\newtheorem{proposition}
{Proposition}
\newtheorem*{propositionA*}{Proposition A}

\theoremstyle{definition} 
{Definition}
\theoremstyle{definition} 
{Example}
\newtheorem*{ex*}{Example}
\theoremstyle{remark} 

\theoremstyle{remark} 
\newtheorem{remark}
{Remark}
\newtheorem*{remark*}{Remark}

\providecommand{\url}[1]{#1}

\renewcommand{\le}{\leqslant}
\renewcommand{\ge}{\geqslant}

\renewcommand{\P}{\operatorname{\mathsf{P}}} 
\newcommand{\E}{\operatorname{\mathsf{E}}}

\newcommand{\ep}{\varepsilon}
\newcommand{\vp}{\varepsilon}

\newcommand{\la}{\lambda}

\newcommand{\Om}{\Omega}

\newcommand{\F}{\mathcal{F}}
\newcommand{\PP}{\mathcal{P}}

\journal{Statistics and Probability Letters}

\begin{document}

\begin{frontmatter}


\title{Improved concentration bounds for sums of independent sub-exponential random variables}
%



\author{Iosif Pinelis}

\address{Department of Mathematical Sciences\\
Michigan Technological University\\
Houghton, Michigan 49931, USA\\
E-mail: ipinelis@mtu.edu}

\begin{abstract}
Known Bernstein-type upper bounds on the tail probabilities for sums of independent zero-mean sub-exponential random variables are improved in several ways at once. The new upper bounds have a certain optimality property.   
\end{abstract}

\begin{keyword}
concentration bounds \sep Bernstein-type bounds \sep sub-exponential random variables \sep sums of independent random variables

\MSC[2010] 
60E15 
\end{keyword}


\end{frontmatter}



%


\section{Known Bernstein-type upper bounds on the tail probabilities for sums of independent zero-mean sub-exponential random variables: Review}\label{review}

Such bounds are used in contemporary data science and, in particular, in analysis of random
matrices; see e.g.\ \cite{versh-HDP,versh-random-matr,chafai-etal}. 

Let us begin here with a few definitions and facts, mostly from the just mentioned sources. 

A function $f\colon[0,\infty)\to[0,\infty)$ is called an Orlicz function if $f$ is convex and 
increasing, with $f(0)=0$ and $\lim_{u\to\infty}f(u)=\infty$. 

For any (real-valued) random variable (r.v.) $X$ defined on a probability space $\PP=(\Om,\F,\P)$ and any Orlicz function $f$, let 
\begin{equation}\label{eq:orlicz}
	\|X\|_f:=\inf\{x>0\colon\E f(|X|/x)\le1\},   
\end{equation}
with the standard convention $\inf\emptyset:=\infty$; 
cf.\ \cite[Section~2.7.1]{versh-HDP}. One may note that, in view of the Fatou lemma,  
\begin{equation}\label{eq:=min}
	\|X\|_f=\min\{x>0\colon\E f(|X|/x)\le1\}   
\end{equation}
provided that $0<\|X\|_f<\infty$. 

The function $\|\cdot\|_f$ is indeed a norm on the vector space of all r.v.'s $X$ defined on the probability space $\PP$ such that $\|X\|_f<\infty$. 

Consider the Orlicz 
function $\psi_1$ defined by the formula
\begin{equation*}
	\psi_1(u):=e^u-1 
\end{equation*}
for $u\in[0,\infty)$. 

A r.v.\ $X$ is called \emph{sub-exponential} if $\|X\|_{\psi_1}<\infty$. This notion of sub-exponentiality (meaning that the tails of the r.v.\ $X$ are lighter than some exponential tails) should not be confused 
with the other common notion of sub-exponentiality, surveyed in \cite{goldie-kluppelberg} (meaning that the tails of the r.v.\ $X$ decay slower than all exponential tails). 


The following proposition is a restatement of \cite[Proposition~2.7.1]{versh-HDP}; cf.\ \cite[Section~5.2.4]{versh-random-matr}. 

\begin{propositionA*}
For any r.v.\ $X$, the 
properties 
\begin{enumerate}[1)]
	\item $\P(|X|\ge x)\le2e^{-x/K_1}$ for all real $x>0$; 
	\item $\E|X|^p\le(K_2p)^p$ for all real $p\ge1$; 
	\item $\E e^{h|X|}\le\exp(K_3 h)$ for all $h\in[0,1/K_3]$; 
	\item $\|X\|_{\psi_1}\le K_4$  
\end{enumerate}
are equivalent to one another in the following sense: for any $i$ and $j$ in the set $\{1,2,3,4\}$, if Property i) holds for some real $K_i>0$, then Property j) holds for some real $K_j>0$ differing from $K_i$ by at most a universal positive real constant factor. 

Moreover, if $\E X=0$, then the property 
\begin{enumerate}[5)]
	\item $\E e^{hX}\le\exp(K_5^2 h^2)$ for all real $h$ 
	such that $|h|\le1/K_5$  
\end{enumerate}
is equivalent, in the similar sense, to each of properties 1)--4). 
\end{propositionA*}

A similar statement, with explicit values of the universal positive real constant factors, was given as 
\cite[Theorem~1.1.5]{chafai-etal}.  

The most important implication in Proposition A is 4)$\implies$5).   
Using this implication, one almost immediately obtains the following Bernstein-type bound, stated as \cite[Theorem 2.8.1]{versh-HDP}; cf.\ \cite[Proposition~5.16]{versh-random-matr}:  

\begin{theoremB*}
Let $S_n:=X_1+\cdots+X_n$, where $X_1,\dots,X_n$ are independent zero-mean sub-exponential r.v.'s. Let 
\begin{equation*}
	B_n^2:=\sum_{i=1}^n\|X_i\|_{\psi_1}^2\quad\text{and}\quad
	M_n:=\max_{i=1}^n\|X_i\|_{\psi_1}. 
\end{equation*}
To avoid trivialities, suppose that $M_n\in(0,\infty)$ or, equivalently, $B_n^2\in(0,\infty)$. 
Then, for some universal positive real constant $c$ and all real $x\ge0$,
\begin{equation*}
	\P(|S_n|\ge x)\le2\exp\Big(-c\min\Big(\frac{x^2}{B_n^2},\frac x{M_n}\Big)\Big). 
\end{equation*}
\end{theoremB*} 

A similar statement, with an explicit value of the constant $c$, was given as 
\cite[Theorem~1.2.7]{chafai-etal}. 


The implication 4)$\implies$5) was obtained in \cite{versh-random-matr,versh-HDP} as a corollary of the implications 4)$\implies$1), 1)$\implies$2), 2)$\implies$5), which are proved as follows: 

4)$\implies$1): Assuming 4), by \eqref{eq:=min} with $f=\psi_1$ and Markov's inequality, for real $x>0$ we get 
\begin{equation*}
	\P(|X|\ge x)\le e^{-x/K_4}\E e^{|X|/K_4}\le2e^{-x/K_4},
\end{equation*}
so that 1) holds with $K_1=K_4$. 

1)$\implies$2): 
Assuming 1), for real $p\ge1$ we get 
\begin{equation*}
	\E|X|^p=\int_0^\infty px^{p-1}\P(|X|\ge x)\,dx
	\le2\int_0^\infty px^{p-1}e^{-x/K_4}\,dx
	=2p!K_4^p\le(2pK_4)^p,
\end{equation*}
so that 1) holds with $K_2=2K_4$. 

2)$\implies$5): Assuming $\E X=0$ and 2), and using the inequality $p!\ge p^p/e^p$ for integers $p\ge1$, for all real $h$ close enough to $0$ we get 
\begin{multline}
	\E e^{hX}\le1+\sum_{p=2}^\infty\frac{|h|^p\E|X|^p}{p!}
	\le1+\sum_{p=2}^\infty |h|^pK_2^pe^p
	=1+\frac{e^2K_2^2h^2}{1-eK_2|h|} \\ 
	=1+\frac{4e^2K_4^2h^2}{1-2eK_4|h|}
\le\exp\frac{4e^2K_4^2h^2}{1-2eK_4|h|}, \label{eq:5}
\end{multline}
so that 5) holds with $K_5=2eK_2=4eK_4$. 

\medskip\hrule\medskip

This proof of the implication 4)$\implies$5) can be streamlined, with an improved result. This can be done as follows: 

Property 4) implies 
\begin{equation}\label{eq:E|X|^p}
	2\ge\E e^{|X|/K_4}\ge\frac{\E|X|^p/K_4^p}{p!}
\end{equation}
for integers $p\ge1$, whence 
\begin{multline}\label{eq:Ee^hX}
	\E e^{hX}\le1+\sum_{p=2}^\infty\frac{|h|^p\E|X|^p}{p!}
	\le1+\sum_{p=2}^\infty 2|h|^pK_4^p
	=1+\frac{2K_4^2 h^2}{1-K_4|h|} \\ 
	\le\exp\frac{2K_4^2 h^2}{1-K_4|h|}; 
\end{multline}
cf.\ \eqref{eq:5}. Here it was beneficial to eliminate Property~1 from the chain \break  4)$\implies$1)$\implies$2)$\implies$5 of implications, used previously to prove the implication 4)$\implies$5). 


\section{
Explicit, improved bounds: using a better Orlicz function}\label{psi11}

In the streamlined derivation at the end of Section~\ref{review} of Property 5) for sub-exponential zero-mean r.v.'s $X$, we used the bound on exponential moments of the r.v.\ $|X|$ given by Property 4) to first bound the absolute moments of $X$ as in \eqref{eq:E|X|^p}, and then we used the Maclaurin series for the exponential function to bound exponential moments of $X$, as in \eqref{eq:Ee^hX}. 

This latter derivation of Property 5) can be further streamlined, with a substantially improved result. This is attained by 
\begin{itemize}
	\item using a better Orlicz function and, at the same time,  
	\item going directly from exponential moments of the r.v.\ $|X|$ to exponential moments of the r.v.\ $X$, without bounding absolute moments of $X$ or using any series (as was done in  Section~\ref{review}). 
\end{itemize}

The better Orlicz function, which will be denoted here by $\psi_{11}$, is defined by the formula 
\begin{equation*}
	\psi_{11}(u):=\psi_1(u)-u=e^u-1-u
\end{equation*}
for $u\in[0,\infty)$. 
Obviously, $\psi_{11}\le\psi_1$ and hence 
\begin{equation}\label{eq:11<1}
	\|X\|_{\psi_{11}}\le\|X\|_{\psi_1}
\end{equation}
for any r.v.\ $X$. 

Also, it is easy to see that a r.v.\ $X$ is sub-exponential if and only if $\|X\|_{\psi_{11}}<\infty$. 
The ``only if'' part of this statement follows immediately from inequality \eqref{eq:11<1}, whereas inequality \eqref{eq:1<11} (stated and proved later in this paper) is an exact quantitative refinement of the ``if'' part. 

%
%
%
%

Now we can state 

\begin{theorem}\label{th:1} 
If $\E X=0$, then 
\begin{equation}\label{eq:th1}
	\E e^{hX}\le1+h^2\|X\|_{\psi_{11}}^2
	\le\exp\{h^2\|X\|_{\psi_{11}}^2\}\le\exp\{h^2\|X\|_{\psi_1}^2\}
\end{equation}
for any real $h$ such that 
\begin{equation*}
	|h|\,\|X\|_{\psi_{11}}\le1 
\end{equation*}
and hence for any real $h$ such that 
\begin{equation*}
	|h|\,\|X\|_{\psi_1}\le1. 
\end{equation*}
\end{theorem} 
Cf.\ \eqref{eq:Ee^hX} and Property 4): the factor $\frac{2}{1-K_4|h|}\ge2$ has now been eliminated. 

The proof of Theorem~\ref{th:1} does not involve bounding absolute moments of $X$ or using any series; instead, it  
is based on 
\begin{lemma}\label{lem:1}
For all real $\la $ with $|\la |\le1$ and all real $x$
\begin{equation*}
	e^{\la x}-1-\la x\le\la ^2 \psi_{11}(|x|). 
\end{equation*}
\end{lemma}

\begin{proof}
Let 
\begin{equation*}
	d(x):=d_\la (x):=\la ^2 \psi_{11}(|x|)-(e^{\la x}-1-\la x). 
\end{equation*}
Then $d''(x)=\la ^2(e^{|x|}-e^{\la x})\ge0$, since $\la x\le|\la |\,|x|\le|x|$, given the condition $|\la |\le1$. So, the function $d$ is convex. Also, 
$d(0)=0$ and $d'(0)=0$. Thus, $d\ge0$, which completes the proof of Lemma~\ref{lem:1}. 
\end{proof}

\begin{proof}[Proof of Theorem~\ref{th:1}]
Without loss of generality (wlog), $\|X\|_{\psi_{11}}\in(0,\infty)$. 
By Lemma~\ref{lem:1} with $\la$ and $x$ replaced by $h\|X\|_{\psi_{11}}$ and $X/\|X\|_{\psi_{11}}$, respectively, we have 
\begin{equation*}
	e^{hX}\le1+hX+h^2\|X\|_{\psi_{11}}^2 \psi_{11}(|X|/\|X\|_{\psi_{11}}).  
\end{equation*}
Taking now the expected values of both sides of this inequality, and recalling the condition $\E X=0$ and formula \eqref{eq:=min}, we obtain the first inequality in \eqref{eq:th1}. The second inequality in \eqref{eq:th1} is trivial, and the third inequality there follows immediately from \eqref{eq:11<1}. 
\end{proof}

As a corollary of Theorem~\ref{th:1}, we have the following improvement of \break 
Theorem~B, without non-specified constants and with the smaller quantities $B_{n,1}^2$ and $M_{n,1}$ in place of $B_n^2$ and $M_n$, respectively: 

\begin{theorem}\label{th:2} 
Let $S_n:=X_1+\cdots+X_n$, where $X_1,\dots,X_n$ are independent zero-mean sub-exponential r.v.'s. Let 
\begin{equation*}
	B_{n,1}^2:=\sum_{i=1}^n\|X_i\|_{\psi_{11}}^2\quad\text{and}\quad
	M_{n,1}:=\max_{i=1}^n\|X_i\|_{\psi_{11}}. 
\end{equation*}
To avoid trivialities, suppose that $M_{n,1}\in(0,\infty)$ or, equivalently, $B_{n,1}^2\in(0,\infty)$. 
Then for all real $x\ge0$ 
\begin{align}
	\P(S_n\ge x)
	&\le\left\{
\begin{alignedat}{2}
	&\exp-\frac{x^2}{4B_{n,1}^2} &&\text{\quad if\quad }x\le\frac{2B_{n,1}^2}{M_{n,1}}, \\ 
	&\exp-\Big(\frac x{M_{n,1}}-\frac{B_{n,1}^2}{M_{n,1}^2}\Big) 
	 &&\text{\quad if\quad }x\ge\frac{2B_{n,1}^2}{M_{n,1}}
\end{alignedat}
	\right. \label{eq:th2-1} \\ 
	&\le\exp\Big(-\min\Big(\frac{x^2}{4B_{n,1}^2},
	\frac x{2M_{n,1}}\Big)\Big). \label{eq:th2-2} 
\end{align}	
In \eqref{eq:th2-1}--\eqref{eq:th2-2}, $B_{n,1}^2$ and $M_{n,1}$ can be replaced by the larger quantities $B_n^2$ and $M_n$, respectively (which will result in upper bounds on $\P(S_n\ge x)$ worse than those in \eqref{eq:th2-1}--\eqref{eq:th2-2}). 
\end{theorem}

Theorem~\ref{th:2} provides upper bounds on the right-tail probability $\P(S_n\ge x)$. Of course, replacing each $X_i$ by $-X_i$, we get the same upper bounds on the left-tail probability $\P(S_n\le -x)$ for $x\ge0$; then, clearly, we get the doubled upper bounds on the two-tail probability $\P(|S_n|\ge x)$ -- cf.\ Theorem B. 

\begin{proof}[Proof of Theorem~\ref{th:2}] 
Using Markov's inequality, the independence of the $X_i$'s, and 
Theorem~\ref{th:1}, for all real $x\ge0$ and all $h\in[0,1/M_{n,1}]$
we have 
\begin{equation*}
	\P(S_n\ge x)\le e^{-hx}\E e^{hS_n}=e^{-hx}\prod_{i=1}^n\E e^{hX_i}
	\le\exp\big(-hx+h^2B_{n,1}^2\big).  	 
\end{equation*}
Minimizing the latter expression in $h\in[0,1/M_{n,1}]$, we get \eqref{eq:th2-1}. Inequality  \eqref{eq:th2-2} is trivial. 

Using the part of Theorem~\ref{th:1} concerning $\|X\|_{\psi_1}$ rather than $\|X\|_{\psi_{11}}$, we see that indeed, in \eqref{eq:th2-1}--\eqref{eq:th2-2}, $B_{n,1}^2$ and $M_{n,1}$ can be replaced by the larger  quantities $B_n^2$ and $M_n$, respectively. 
\end{proof}

%

\begin{proposition}\label{prop:gain}
(Cf.\ \eqref{eq:11<1}.) For any r.v.\ $X$, 
\begin{equation}\label{eq:1<11}
	\|X\|_{\psi_1}\le C\|X\|_{\psi_{11}},
\end{equation}
where 
\begin{equation}\label{eq:C}
	C:=-\frac{W_{-1}(-e^{-2})+2}{\ln2}=1.653\ldots=\sqrt{2.734\dots}
\end{equation}
and $W_{-1}$ is the $(-1)$th branch of the Lambert $W$ function \cite{knuth96}, so that for all real $z$ and $w$ one has $w=W_{-1}(z)\in(-\infty,-1)\iff z=we^w\in(-e^{-1},0)$. 

Inequality \eqref{eq:1<11} turns into the equality if the r.v.\ $|X|$ is a constant. 
\end{proposition}

\begin{proof}[Proof of Proposition~\ref{prop:gain}]
By rescaling, wlog $\|X\|_{\psi_1}=1$, that is, $\E Y=2$, where $Y:=e^{|X|}$. Letting then for brevity $x:=\|X\|_{\psi_{11}}$, by \eqref{eq:11<1} and the assumption $\|X\|_{\psi_1}=1$ we get $x\in(0,1]$.

Also, by \eqref{eq:=min},
\begin{equation*}
	1=\E\psi_{11}(|X|/x)=\E g_t(Y),
\end{equation*}
where $g_t(y):=y^t-1-t\ln y$ for real $y>0$ and $t:=1/x\ge1$. 
Since the function $g_t$ is convex and $\E Y=2$, by Jensen's inequality we get 
\begin{equation}\label{eq:1>g}
	1=\E g_t(Y)\ge g_t(2)=2^t-1-t\ln 2=\psi_{11}(u),
\end{equation}
where $u:=t\ln2$. Since the function $\psi_{11}$ is increasing, it follows that 
\begin{equation*}
	u\le u_*:=\psi_{11}^{-1}(1)=-(W_{-1}(-e^{-2})+2). 
\end{equation*}
Recalling that $\|X\|_{\psi_{11}}=x=1/t=(\ln2)/u$ and $\|X\|_{\psi_1}=1$, we see that inequality \eqref{eq:1<11} is equivalent to the inequality $u\le u_*$. 


If now $|X|$ is a constant, then wlog 
$X=\ln2$ a.s., so that $\|X\|_{\psi_1}=1$ and also the inequality in \eqref{eq:1>g} turns into the equality; hence, inequality \eqref{eq:1<11} turns into the equality. 

This completes the proof of Proposition~\ref{prop:gain}. 
\end{proof}

\begin{remark}\label{rem:11 vs 1}
The gain provided by using the norm $\|\cdot\|_{\psi_{11}}$ instead of $\|\cdot\|_{\psi_1}$ can be quite substantial. For instance, if $X_1,\dots,X_n$ are independent r.v.'s with the same centered uniform distribution, then $B_n^2/B_{n,1}^2=2.215\dots$. Since $B_n^2$ and $B_{n,1}^2$ appear as factors in the exponents, the $\psi_1$-based bound 
$\exp-\frac{x^2}{4B_2^2}$ on $\P(S_n\ge x)$ and the corresponding improved, $\psi_{1,1}$-based bound $\exp-\frac{x^2}{4B_{n,1}^2}$ (as in \eqref{eq:th2-1})
can be as different as $0.1$ and $0.1^{2.215\dots}=0.006091\dots$, respectively. 

It follows from Proposition~\ref{prop:gain} that the greatest possible value of the improvement ratio $B_n^2/B_{n,1}^2
=\sum_{i=1}^n\|X_i\|_{\psi_1}^2/\sum_{i=1}^n\|X_i\|_{\psi_{11}}^2$ over all independent zero-mean r.v.'s $X_1,\dots,X_n$ 
with $\|X_i\|_{\psi_1}\in(0,\infty)$ for all $i$ is $C^2=2.734\dots$, attained when $\P(X_i=1)=\P(X_i=-1)=1/2$ for all $i$ (where $C$ is as in \eqref{eq:C}), and then the value  $0.1^{2.215\dots}=0.006091\dots$ in the above paragraph is replaced by the yet smaller value $0.1^{2.734\dots}=0.001843\dots$. 
\end{remark}

The function $\psi_{11}$ appeared more or less explicitly in a number of papers, including \cite[formula~(2a)]{bennett}, \cite[formula~(1)]{pin-sakh}, \cite[Theorem~3.1]{pin94}, 
\cite[proof of Theorem~2.10]{BLM}. However, $\psi_{11}$ does not seem to have been
considered specifically as an Orlicz function 
in the context of sums of sub-exponential r.v.'s. 


\section{Yet another improvement: using a smaller threshold level in the definition of the Orlicz norm}\label{ep}

There is no compelling reason to compare $\E f(|X|/x)$ with the specific constant $1$ in the definition \eqref{eq:orlicz} of the Orlicz norm $\|X\|_f$; instead of $1$, one can use there any other positive real constant, which may be referred to as the threshold level. Moreover, we shall see that in typical applications it is advantageous to let the threshold level be small. 

Accordingly, take any real number $\vp>0$, which may be thought of as small. Generalize  definition \eqref{eq:orlicz} as follows:  
\begin{equation}\label{eq:orlicz-ep}
	\|X\|_{f;\vp}:=\inf\{x>0\colon\E f(|X|/x)\le\vp\}.    
\end{equation}
So, $\vp$ is the threshold level of the Orlicz norm $\|\cdot\|_{f;\vp}$, which is still a norm. Obviously, 
$\|X\|_{f;1}=\|X\|_f$ and $\|X\|_{f;\ep}=\|X\|_{f/\vp}$. 
Moreover, for any threshold level $\vp>0$, a r.v.\ $X$ is sub-exponential iff $\|X\|_{\psi_1;\vp}<\infty$ iff $\|X\|_{\psi_{11};\vp}<\infty$. 

Theorems~\ref{th:1} and \ref{th:2} get now respectively generalized as well (with almost the same proofs
): 

\begin{theorem1.1ep*}\label{th:1-ep} 
If $\E X=0$, then 
\begin{equation}\label{eq:th1-ep}
	\E e^{hX}\le1+h^2\vp\|X\|_{\psi_{11};\vp}^2
	\le\exp\{h^2\vp\|X\|_{\psi_{11};\vp}^2\}\le\exp\{h^2\vp\|X\|_{\psi_1;\vp}^2\}
\end{equation}
for any real $h$ such that 
\begin{equation*}
	|h|\,\|X\|_{\psi_{11};\vp} \le1 
\end{equation*}
and hence for any real $h$ such that 
\begin{equation*}
	|h|\,\|X\|_{\psi_1;\vp}\le1. 
\end{equation*}
\end{theorem1.1ep*} 

\begin{theorem1.2ep*}
Let $S_n:=X_1+\cdots+X_n$, where $X_1,\dots,X_n$ are independent zero-mean sub-exponential r.v.'s. Let 
\begin{equation*}
	B_{n,1;\vp}^2:=\vp\sum_{i=1}^n\|X_i\|_{\psi_{11};\vp}^2\quad\text{and}\quad
	M_{n,1;\vp}:=\max_{i=1}^n\|X_i\|_{\psi_{11};\vp}. 
\end{equation*}
To avoid trivialities, suppose that $M_{n,1;\vp}\in(0,\infty)$ or, equivalently, $B_{n,1;\vp}^2\in(0,\infty)$. 
Then for all real $x\ge0$ 
\begin{align}
	\P(S_n\ge x)
	&\le\left\{
\begin{alignedat}{2}
	&\exp-\frac{x^2}{4B_{n,1;\vp}^2} &&\text{\quad if\quad }x\le\frac{2B_{n,1;\vp}^2}{M_{n,1;\vp}}, \\ 
	&\exp-\Big(
	\frac x{M_{n,1;\vp}}-\frac{B_{n,1;\vp}^2}{M_{n,1;\vp}^2}\Big) 
	 &&\text{\quad if\quad }x\ge\frac{2B_{n,1;\vp}^2}{M_{n,1;\vp}}
\end{alignedat}
	\right. \label{eq:th2-ep1} \\ 
	&\le\exp\Big(-\min\Big(\frac{x^2}{4B_{n,1;\vp}^2},
	\frac x{2M_{n,1;\vp}}\Big)\Big). \label{eq:th2-ep2} 
\end{align}	  
\end{theorem1.2ep*}

Of course, for any r.v.\ $X$ we have the generalization $\|X\|_{\psi_{11};\vp}\le\|X\|_{\psi_1;\vp}$ of inequality \eqref{eq:11<1}, and therefore we may replace $B_{n,1;\vp}^2$ and $M_{n,1;\vp}$ by their larger counterparts 
\begin{equation}\label{eq:n,vp}
B_{n;\vp}^2:=\vp\sum_{i=1}^n\|X_i\|_{\psi_1;\vp}^2\quad\text{and}\quad
	M_{n;\vp}:=\max_{i=1}^n\|X_i\|_{\psi_1;\vp},  
\end{equation}
defined using $\psi_1$ instead of $\psi_{11}$. 

In \cite[Theorem 1.2.7]{chafai-etal}, another upper bound on $\P(S_n\ge x)$ for $x\ge0$ was obtained, which can be rewritten in our notation as 
\begin{equation*}
	\exp\Big(-\min\Big(\frac{x^2}{4c_2B_{n;\vp}^2},
	\frac x{2c_1M_{n;\vp}}\Big)\Big)
\end{equation*}
for $\vp=e-1$, 
where $c_1:=2e-1>1$ and $c_2:=\frac{2e-1}{2e-2}>1$. Clearly, for all real $x>0$, this bound from \cite{chafai-etal} is worse (that is, greater) than 
the upper bound 
\begin{equation*}
	\exp\Big(-\min\Big(\frac{x^2}{4B_{n;\vp}^2},
	\frac x{2M_{n;\vp}}\Big)\Big), 
\end{equation*}
which latter is the weakened version of the upper bound in \eqref{eq:th2-ep2}, with $\psi_1$ in place of $\psi_{11}$. 

Actually, as noted in the beginning of this section, with the Orlicz function $\psi_{11}$, it is usually beneficial to let the threshold level $\vp$ be small, rather than greater than $1$.  
Moreover, if $\vp$ is small, then the contrast between $\|X\|_{\psi_{11};\vp}$ and $\|X\|_{\psi_1;\vp}$ is much starker than that allowed (for $\vp=1$) by Proposition~\ref{prop:gain}:  

\begin{proposition}\label{prop:gain-ep}
Take any sub-exponential r.v.'s $X$ such that $0<\|X\|_{\psi_1}<\infty$. Let $\vp\downarrow0$. Then  
\begin{equation}\label{eq:1-ep}
	\|X\|_{\psi_1;\vp}\sim\frac{\E|X|}\vp, 
\end{equation}
\begin{equation}\label{eq:11-ep}
	\|X\|_{\psi_{11};\vp}\sim\sqrt{\frac{\E X^2}{2\vp}}, 
\end{equation}
so that $\|X\|_{\psi_{11};\vp}=o(\|X\|_{\psi_1;\vp})$. 
\end{proposition}

\begin{proof}[Proof of Proposition~\ref{prop:gain-ep}] 
Note that 
\begin{equation}\label{eq:1,ep,=}
	\|X\|_{\psi_1;\vp}=\frac1\vp\,\inf\{z>0\colon\E R_{1;\vp}\le1\},
\end{equation}
where 
\begin{equation*}
	R_{1;\vp}:=\frac{e^{\vp|X|/z}-1}\ep\to\frac{|X|}z;  
\end{equation*}
Here and in what follows in this proof, $z$ is an arbitrary positive real number. 
Moreover, 
\begin{equation*}
	0\le R_{1;\vp}\le\frac{|X|}z\,e^{\vp|X|/z}.
\end{equation*}
Therefore and because the r.v.\ $X$ is sub-exponential, one can use use the dominated convergence theorem to conclude that $\E R_{1;\vp}\to\frac{\E|X|}z$. Now \eqref{eq:1-ep} follows from \eqref{eq:1,ep,=}. 

Similarly, 
\begin{equation}\label{eq:11,ep,=}
	\|X\|_{\psi_{11;\vp}}=\frac1{\sqrt\vp}\,\inf\{z>0\colon\E R_{11;\vp}\le1\},
\end{equation}
where 
\begin{equation*}
	R_{11;\vp}:=\frac{e^{\sqrt\vp\,|X|/z}-1-\sqrt\vp\,|X|/z}\ep\to\frac{X^2}{2z^2}. 
\end{equation*}
Moreover, 
\begin{equation*}
	0\le R_{11;\vp}\le\frac{X^2}{2z^2}\,e^{\sqrt\vp\,|X|/z}.
\end{equation*}	
So, by the dominated convergence theorem, $\E R_{11;\vp}\to\frac{\E X^2}{2z^2}$. Now \eqref{eq:11-ep} follows from \eqref{eq:11,ep,=}. 
\end{proof}

Almost immediately from Theorem 2.$\boldsymbol{\vp}$ 
and Proposition~\ref{prop:gain-ep} we get 
\begin{corollary}\label{cor:1}
Let $X_1,X_2,\dots$ be independent identically distributed zero-mean sub-exponential r.v.'s such that $0<\|X_1\|_{\psi_1}<\infty$. For each natural $n$, let  $S_n:=X_1+\cdots+X_n$. Then
\begin{equation}\label{eq:cor1}
\begin{aligned}
	\P(S_n\ge x)&\le\exp-\frac{x^2}{(2+o(1))\,n\E X_1^2}
\end{aligned}	 
\end{equation}
whenever $n$ and $x$ vary simultaneously so that 
$0\le x/n\to0$. (Here we may, but do not have to, insist that $n\to\infty$.)
\end{corollary}

Indeed, by Proposition~\ref{prop:gain-ep}, in the conditions of Corollary~\ref{cor:1} we have $B_{n,1;\vp}^2=n\vp\|X_1\|_{\psi_{11};\vp}^2\sim n\E X_1^2/2$ and 
$M_{n,1;\vp}=\|X_1\|_{\psi_{11};\vp}\sim\sqrt{\E X_1^2/(2\vp)}$ as $\vp\to0+$, so that 
\begin{equation}
	\frac{2B_{n,1;\vp}^2}{M_{n,1;\vp}}\sim n\sqrt{\E X_1^2}\sqrt{2\vp}. 
\end{equation}
Letting now $\vp=(x/n)^2/\E X_1^2$, by the condition $x/n\to0$ we will have $\vp\to0+$, and the condition $x\le\frac{2B_{n,1;\vp}^2}{M_{n,1;\vp}}$ in \eqref{eq:th2-ep1} will be satisfied eventually, that is, for all small enough values of $x/n\ge0$. 
So, by Theorem 2.$\boldsymbol{\vp}$, we will have 
\begin{equation}
	\P(S_n\ge x)
	\le\exp-\frac{x^2}{4B_{n,1;\vp}^2}=\exp-\frac{x^2}{(2+o(1))\,n\E X_1^2}, 
\end{equation}
which proves \eqref{eq:cor1}. 

Corollary~\ref{cor:1}, 
with the optimal constant $2$ in the denominator in \eqref{eq:cor1}, cannot apparently be obtained without using the generalized Orlicz norm $\|\cdot\|_{f;\vp}$ with $\vp\downarrow0$ or without using the Orlicz function $\psi_{11}$ instead of $\psi_1$. 

E.g., in the case when $\P(X_i=1)=\P(X_i=-1)=1/2$ for all $i$ we can say the following:  
\begin{enumerate}[(i)]
		\item Using Theorem 2.$\boldsymbol{\vp}$ with $\vp=1$ (that is, using Theorem~\ref{th:2}), we would only get the constant $4\|X_1\|_{\psi_{11;1}}^2=4\|X_1\|_{\psi_{11}}^2=\frac{4}{\left(W_{-1}\left(-e^{-2}\right)+2\right){}^2}=3.044\dots$ in place of $2+o(1)$ in the denominator in \eqref{eq:cor1}; here, again, $W_{-1}$ is the $(-1)$th branch of the Lambert $W$ function.
	\item Note that here $\|X_1\|_{\psi_1;\vp}=g(\vp):=1/\ln(1+\vp)\sim1/\vp$ as $\vp\downarrow0$. (By Proposition~\ref{prop:gain-ep}, the asymptotic $\|X\|_{\psi_1;\vp}\sim\E|X|/\vp$ actually holds in general, for any any sub-exponential r.v.\ $X$ such that $0<\|X\|_{\psi_1}<\infty$.) So, $B_{n;\vp}^2=n\vp g(\vp)^2\to\infty$ as $\vp\downarrow0$, where $B_{n;\vp}^2$ is as defined in \eqref{eq:n,vp}. Thus, if we weaken Theorem~2.$\boldsymbol{\vp}$ by replacing there  $\psi_{11}$ by $\psi_1$, then instead of $2+o(1)$ in \eqref{eq:cor1} we will get a quantity going to $\infty$ as $\vp\downarrow0$. We see that letting $\vp\downarrow0$ is not at all helpful when $\psi_1$ is used, rather than $\psi_{11}$. 
The optimal choice of $\vp$ in this weakened, $\psi_1$ version of Theorem 2.$\boldsymbol{\vp}$ is the minimizer $\vp_*:=
e^{W\left(-2e^{-2}\right)+2}-1=3.921\dots$ of $\vp g(\vp)^2$ in $\vp>0$,  
and with this choice of $\vp$ 
%
we would only get the constant $4\vp_* g(\vp_*)^2=6.176\dots$ in place of $2+o(1)$ in the denominator in \eqref{eq:cor1}. 
(Here $W$ is the $0$th branch of the Lambert $W$ function, so that for all real $z$ and $w$ one has $w=W(z)\in(-1,\infty)\iff z=we^w\in(-e^{-1},\infty)$.) 
\end{enumerate}


However, exponential bounds 
of forms similar to that of the bound in \eqref{eq:cor1}, with $2+o(1)$ in the denominator, 
were previously obtained in somewhat different settings -- see \cite{MR821760} and references there. 

\section{Illustration of improvements}\label{pic}
Graphs of the upper bound on $\P(S_n\ge x)$ in \eqref{eq:th2-ep1} for independent r.v.'s $X_1,\dots,X_n$ with $\P(X_i=1)=\P(X_i=-1)=1/2$ are presented in Figure~\ref{fig:pic} 
%
for $x\in[0,3\sqrt n\,]$ with $\vp=0.1$ (solid black), $\vp=0.3$ (solid gray), $\vp=1$ (solid light gray) -- for $n=10$ (left) and $n=100$ (right). Also shown in Figure~\ref{fig:pic} are the corresponding (dashed) graphs of the bound obtained from the bound in \eqref{eq:th2-ep1} by replacing there $\|X_i\|_{\psi_{11};\vp}$ with $\|X_i\|_{\psi_1;\vp}$.

\begin{figure}[htbp] 
	\centering
		\includegraphics[width=1.00\textwidth]{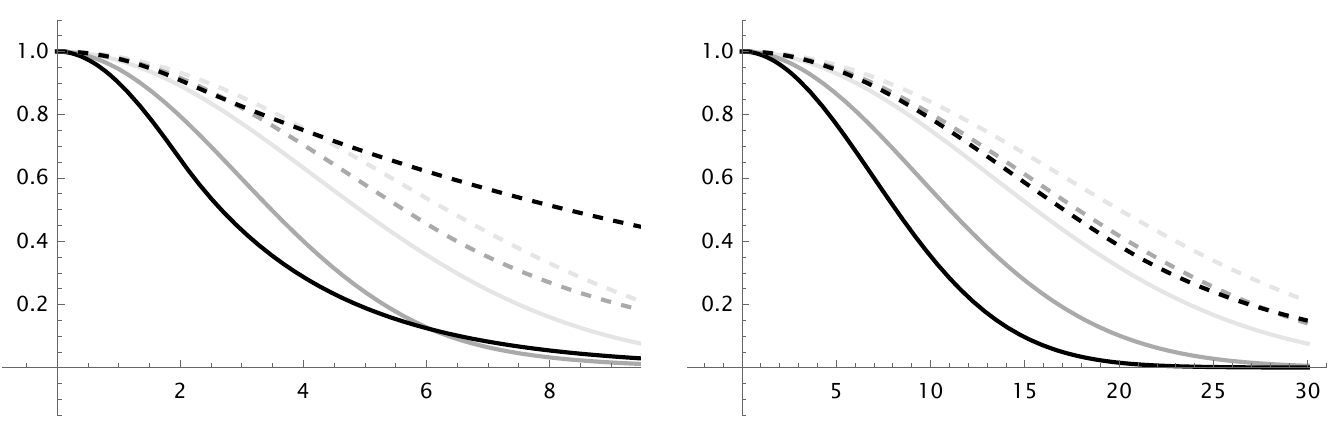}
	\caption{Graphs of the various bounds. 
	} 
	\label{fig:pic}
\end{figure}

These graphs 
strongly suggest that (i) replacing the Orlicz function $\psi_1$ by $\psi_{11}$ results in quite substantial improvement of the upper bounds on $\P(S_n\ge x)$, especially in the large-deviation zone, and (ii) replacing the threshold level $1$ in the definition \eqref{eq:orlicz} of the Orlicz norm by a smaller level $\vp>0$ is also generally beneficial. 

\section{Conclusion}\label{conclution}

In this note, the following has been done:  
\begin{itemize}
	\item The derivation of Bernstein-type upper bounds on the tail probabilities for sums of independent zero-mean sub-exponential r.v.'s $X_i$ has been streamlined, by avoiding intermediate steps of bounding tails and absolute moments of the individual summands $X_i$.  
	\item 
	The resulting bounds do not contain any non-explicit constants, in contrast with some previously known bounds, and the new bounds improve corresponding existing bounds with explicit constants -- uniformly, for all values of the relevant parameters. 
	\item 
Moreover, the Orlicz function $0\le u\mapsto\psi_1(u)=e^u-1$, commonly used so far for sub-exponential r.v.'s, has been replaced by the Orlicz function $0\le u\mapsto\psi_{11}(u)=e^u-1-u$, 
which led to substantially improved bounds, presented in Theorem~\ref{th:2}. This improvement effect can be explained as follows: In contrast with $\psi_1$, the Orlicz function $\psi_{11}$ is asymptotically quadratic near $0$, which allows us to bound effectively, in Lemma~\ref{lem:1}, the asymptotically quadratic (for $x$ near $0$) expression $e^{\la x}-1-\la x$ by means of the asymptotically quadratic (for $x$ near $0$) expression $\psi_{11}(x)=e^{|x|}-1-|x|$.  
	\item A further substantial improvement was obtained in Theorem 2.${\vp}$  
by using smaller threshold levels $\vp$ -- 
instead of the commonly used threshold level $1$, as in the definition \eqref{eq:orlicz} of the Orlicz norm. 
Using smaller values of $\vp$ helps because the mentioned asymptotically quadratic behavior of $\psi_{11}(|x|)$ for $x$ near $0$ is then captured more effectively, since the smallest values of any Orlicz function are taken in a neighborhhod of $0$. 
Thus obtained bounds have a certain, previously unavailable optimality property, exhibited in Corollary~\ref{cor:1}. 
\end{itemize}

%
%

Note finally that having explicit and accurate bounds is of obvious importance if the theory is to ever be applied to real data. 


\bibliographystyle{amsplain}

\bibliography{C:/Users/ipinelis/Documents/pCloudSync/mtu_pCloud_02-02-17/bib_files/citations04-02-21}


\end{document}